\documentclass{amsart}

\usepackage{mathrsfs}
\usepackage{epsfig}
\usepackage{subfigure}

\usepackage{caption}
\usepackage{multirow}
\usepackage{amsmath,amsfonts,amssymb}
\usepackage{fancyhdr}
\usepackage[top=1in, bottom=1in, left=1.25in, right=1.25in]{geometry}
\def \a {\alpha}
\def \b {\beta}
\def \p {\partial}
\def \l {\lambda}
\def \tf {\tilde{f}}
\def \tui {\tilde{u}_{N,0}}
\def \tuN {\tilde{u}_N}
\def \sech {\mathrm{sech}}
\def \tanh {\mathrm{tanh}}

\def \R {\mathbb{R}}
\def \CRN {\mathcal{R}_N}
\def \N {\mathbb{N}}

\begin{document}
\renewcommand{\theequation}{\arabic{section}.\arabic{equation}}
\fancyhf{}

\fancyhead[LO]{Time-dependent Hermite-Galerkin spectral method} \cfoot{\thepage}

\title{Time-dependent Hermite-Galerkin spectral method and its applications}

\author{Xue Luo}
\author{Shing-Tung Yau}
\author{Stephen S.-T. Yau}

\address{School of Mathematics and System Sciences, Beihang University, Beijing, P. R. China 100191. email: xluo@buaa.edu.cn}
\address{Department of Mathematics, Harvard University, Cambridge, MA 02138, US. email: yau@math.harvard.edu}
\address{Department of Mathematical Sciences, Tsinghua
University, Beijing, P. R. China, 100084. email: yau@uic.edu}

\newtheorem{definition}{Definition}
\renewcommand{\thedefinition}{\arabic{section}.\arabic{definition}}
\newtheorem{proposition}{Proposition}
\renewcommand{\theproposition}{\arabic{section}.\arabic{proposition}}
\newtheorem{theorem}{Theorem}
\renewcommand{\thetheorem}{\arabic{section}.\arabic{theorem}}
\newtheorem{lemma}{Lemma}
\renewcommand{\thelemma}{\arabic{section}.\arabic{lemma}}
\newtheorem{corollary}{Corollary}
\renewcommand{\thecorollary}{\arabic{section}.\arabic{corollary}}
\newtheorem{remark}{Remark}
\renewcommand{\theremark}{\arabic{section}.\arabic{remark}}
\renewcommand{\thefigure}{\arabic{section}.\arabic{figure}}

\begin{abstract}
	A time-dependent Hermite-Galerkin spectral method (THGSM) is investigated in this paper for the nonlinear convection-diffusion equations in the unbounded domains. The time-dependent scaling factor and translating factor are introduced in the definition of the generalized Hermite functions (GHF). As a consequence, the THGSM based on these GHF has many advantages, not only in theorethical proofs, but also in numerical implementations. The stability and spectral convergence of our proposed method have been established in this paper. The Korteweg-de Vries-Burgers (KdVB) equation and its special cases, including the heat equation and the Burgers' equation, as the examples, have been numerically solved by our method. The numerical results are presented, and it surpasses the existing methods in accuracy. Our theoretical proof of the spectral convergence has been supported by the numerical results. 
\end{abstract}


\maketitle

\section{Introduction}

\setcounter{equation}{0}

Many scientific and engineering problems are naturally modelled in the unbounded domains. One way to numerically solve the problems is to restrict the model equation in some bounded domain and artificially impose some boundary condition cleverly. Whereas this introduces errors even before the implementation of the numerical scheme. Another more suitable way is to use the spectral approaches employing orthogonal systems in unbounded domain, such as using Laguerre polynomials for the problems in semi-bounded or exterior domains \cite{GS,KS}, and using Hermite polynomials for those in the whole space \cite{A,FGT,FK,GX,MST}. 

Although the freedom from artificial boundary condition is very attractive, the Hermite spectral method (HSM) is only widely studied in the recent decade, due to its poor resolution without the appropriate scaling factor. Gottlieb and Orszag  \cite{GO} claim that to resolve $M$ wavelength of $\sin{x}$, it requires nearly $M^2$ Hermite polynomials. The Hermite functions, defined as $\displaystyle\left\{H_n(x)e^{-x^2}\right\}_{n=0}^\infty$, have the same deficiency as the polynomials $\displaystyle\left\{H_n(x)\right\}_{n=0}^\infty$. The importance of the scaling factor has been discussed in \cite{T,SH}. It has been shown in \cite{B} that the scaling factor should be selected according to the truncated modes $N$ and the asymptotical
behavior of the function $f(x)$, as $|x|\rightarrow\pm\infty$. The optimal scaling factor is still an open problem, even in the case that $f(x)$ is given explicitly, to say nothing of the exact solution to a differential equation, which is in general unknown a-priorily. Recently, during the study of using the HSM to solve the nonlinear filtering problems, Yau and the author gave a practical strategy in \cite{LY} to pick the appropriate scaling factor and the corresponding truncated mode for at least the most commonly used types of functions, i.e. the Gaussian type and the super-Gaussian type. Thanks to this guideline, the Hermite-Galerkin spectral method (HGSM) becomes implementable.

In the literature of solving partial differential equations in unbounded domains using HSM, nearly all the schemes are not direct Galerkin ones. As far as the author knows, there are at least two possible reasons:
\begin{enumerate}
	\item The lack of the practical guidelines of choosing appropriate scaling factor before \cite{LY} makes the direct Galerkin method infeasible.  
	\item When directly applying traditional definition of Hermite functions (i.e., $\{H_n(x)e^{-x^2}\}_{n=0}^\infty$) to second-order differential equations, it is found in \cite{G} that the stiff matrix is of nonsymmetric bilinear form, which has no property of coercity. In other word, the stability can not be established by using the classical energy method.
\end{enumerate}

To overcome the obstacles above, Funaro and Kavian \cite{FK} first consider the use of the Hermite polynomials to approximate the solutions of some diffusion evolution equations in unbounded domains. The variable transformation technique is introduced to get better resolution. Later, Guo \cite{G}, Guo and Xu \cite{GX} developed the Hermite polynomial spectral and pseudo-spectral methods, where the transformation $U=e^{-x^2}V$ is used, and then $V$ is approximated by the Hermite polynomials. Ma, Sun and Tang \cite{MST}, Ma and Zhao \cite{MZ} introduced a time-dependent parameter to stabilize the scheme, which is based on the traditional defined Hermite functions. However, no discussion was given on how to choose such parameter for the particular problems. 

The aim of this paper is to develop a time-dependent Hermite-Galerkin spectral method (THGSM) to approximate the solution to the nonlinear convection-diffusion equations with high accuracy. The time-dependence is reflected in the definition of the generalized Hermite functions (GHF), where the scaling factor and the translating factor are the functions of time. The choice of the time-dependent scaling factor can follow the guidelines in \cite{LY}, while the time-dependent translating factor mainly deals with the time-shifting of the solution, see examples in section 4.3. The advantages of our THGSM are the following:
\begin{enumerate}
	\item It is a direct Galerkin scheme, which can be implemented straight forward. And the resulting stiffness matrix of the second-order differential equations are of nice properties. For example, it is tri-diagonal, symmetric and diagonally dominant in the linear case, i.e. $g(u)\equiv0$ in \eqref{1D KdVB with source}; it is symmetric for the Burgers' equation, i.e. $g(u)=\frac u2$ in \eqref{1D KdVB with source}. 
	\item The proofs of stability and spectral convergence are greatly simplified, thanks to the definition of GHF. They are analyzed in the $L^2$ space, instead of the weighted one as in \cite{MST}.
	\item From the numerical simulations in section 4, our scheme outperforms nearly all the existing methods in accuracy.
\end{enumerate}

An outline of the paper is as follows. In section 2, we give the definition of GHF and its properties. For the readers' convenience, we include the proof of the error estimate of the orthogonal projection. Our TGHSM to solve the nonlinear convection-diffusion equations is introduced in section 3. The stability analysis in the sense of \cite{G} and the spectral convergence are shown there. Section 4 is devoted to the numerical simulations, where we compared the numerical results with those obtained by other methods in some benchmark equations.  

\section{Generalized Hermite functions (GHF)}

\setcounter{equation}{0}

We introduce the GHF and derive some properties which are inherited from the physical Hermite polynomials. For the sake of completeness, we give the proof of the convergence rate of the orthogonal approximation.

\subsection{Notations and Preliminaries}

Let $L^2(\mathbb{R})$ be the Lebesgue space, which equips with the
norm $||\cdot||=(\int_{\mathbb{R}}|\cdot|^2 dx)^{\frac12}$ and the
scalar product $\langle\cdot,\cdot\rangle$.

Let $\mathcal{H}_n(x)$  be the physical Hermite polynomials given by
$\mathcal{H}_n(x)=(-1)^ne^{x^2}\partial_x^ne^{-x^2}$, $n\geq0$. More
practically, the three-term recurrence
\begin{align}\label{recurrence}
    \mathcal{H}_0\equiv1;\quad \mathcal{H}_1(x)=2x;\quad\textup{and}\quad
    \mathcal{H}_{n+1}(x)=2x\mathcal{H}_n(x)-2n\mathcal{H}_{n-1}(x).
\end{align}
is more handy in implementation. One of the well-known and useful
fact of Hermite polynomials is that they are mutually orthogonal
with the weight $w(x)=e^{-x^2}$. We define the time-dependent GHF as
\begin{align}\label{new hermite}
    H_n^{\alpha,\beta}(x,t)=\left(\frac{\alpha(t)}{2^nn!\sqrt\pi}\right)^{\frac12}\mathcal{H}_n[\alpha(t)(x-\beta(t))]e^{-\frac12\alpha^2(t)[x-\beta(t)]^2},
\end{align}
for $n\geq0$, where $\alpha(t)>0$, $\b(t)$, for $t\in[0,T]$, are functions of time. For the conciseness of notation, let us denote $d(n)=\sqrt{\frac n2}$. And if no confusion will arise, in the sequel we omit the $t$ in $\a(t)$ and $\b(t)$. It is readily to derive the following properties for
the GHF (\ref{new hermite}):
\begin{enumerate}
    \item [$\blacksquare$] \label{orthogonal of H_n^alpha} At each time $t>0$, $\{H_n^{\alpha,\beta}(\cdot,t)\}_{n\in\mathbb{Z^+}}$ form the orthogonal basis of $L^2(\mathbb{R})$, i.e.
        \begin{align}\label{orthogonal}
            \int_{\mathbb{R}}H_n^{\alpha,\beta}(x,t)H_m^{\alpha,\beta}(x,t)dx=\delta_{nm},
        \end{align}
    where $\delta_{nm}$ is the Kronecker function.
    \item [$\blacksquare$]\label{eigenvalue} $H_n^{\alpha,\beta}(\cdot,t)$ is the $n^{\textup{th}}$
    eigenfunction of the following Strum-Liouville problem
        \begin{align*}
            \mathcal{L}^{\a,\b}u(x,t)=\lambda_n(t)u(x,t),
        \end{align*}
where 
\[
	\mathcal{L}^{\a,\b}(\circ)=-e^{\frac12\alpha^2(x-\beta)^2}\partial_x(e^{-\alpha^2(x-\beta)^2}\partial_x(e^{\frac12\alpha^2(x-\beta)^2}\circ))
\]
with the corresponding eigenvalue $\lambda_n(t)=2\alpha^2(t)n$.
    \item [$\blacksquare$] By convention, $H_n^{\alpha,\beta}\equiv0$, for $n<0$. For
    $n\geq0$, the three-term recurrence is inherited from the
    Hermite polynomials:
    \begin{align}\label{recurrence for RT hermite function}\notag
       xH_n^{\alpha,\beta}(x,t) =&
                \frac{d(n+1)}{\alpha}H_{n+1}^{\alpha,\beta}(x,t)
                +\beta H_n^{\alpha,\beta}(x,t)
                +\frac{d(n)}{\alpha}H_{n-1}^{\alpha,\beta}(x,t);\\
        \textup{or}\quad
      \a(x-\b)H_n^{\a,\b}(x,t)=&d(n+1)H_{n+1}^{\a,\b}(x,t)+d(n)H_{n-1}^{\a,\b}(x,t).
    \end{align}
    \item [$\blacksquare$] The derivative of $H_n^{\alpha,\beta}(x,t)$ with respect to $x$
    and $t$
    \begin{align}\label{derivative_x}\notag
        \partial_xH_n^{\alpha,\beta}(x,t)
        =&-\frac12\sqrt{\lambda_{n+1}}H_{n+1}^{\alpha,\beta}(x,t)
            +\frac12\sqrt{\lambda_n}H_{n-1}^{\alpha,\beta}(x,t)\\
       =&-d(n+1)\a H_{n+1}^{\a,\b}(x,t)+d(n)\a H_{n-1}^{\a,\b}(x,t).
\end{align}
    \item [$\blacksquare$] The ``orthogonality" of $\left\{\partial_xH_n^{\alpha,\beta}(\cdot,t)\right\}_{n\in\mathbb{Z^+}}$
    \begin{align}\label{orthogonality of derivative}
    \int_{\mathbb{R}}\partial_xH_n^{\alpha,\beta}\partial_xH_m^{\alpha,\beta}dx =
    \left\{ \begin{aligned}
        \a^2\left[d^2(n+1)+d^2(n)\right],\quad&\textup{if}\
    m=n;\\
        -\a^2d(l+1)d(l+2),\quad& l=\min\{n,m\},\ \textup{if}
    \ |n-m|=2;\\
       0,\quad&\textup{otherwise}.
    \end{aligned} \right.
\end{align}
\end{enumerate}

\begin{lemma}[Derivative with respect to $t$]\label{lemma-derivative_t}
	\begin{align}\label{eqn-derivative_t}\notag
		\partial_tH_n^{\a,\b}(x,t)
	=&-\frac{\a'}\a d(n+1)d(n+2)H_{n+2}^{\a,\b}(x,t)+\a\b'd(n+1)H_{n+1}^{\a,\b}\\
	&-\a\b'd(n)H_{n-1}^{\a,\b}(x,t)+\frac{\a'}\a d(n)d(n-1)H_{n-2}^{\a,\b}(x,t),
	\end{align}
where $d(n)=\sqrt{\frac n2}$ and $\a,\b$ are the functions of $t$.
\end{lemma}
\begin{proof}
	Through direct computations, we have
\begin{align*}
	\p_tH_n^{\a,\b}(x,t)\overset{\eqref{new hermite}}=&\frac{\a'}{2\a}H_n^{\a,\b}(x,t)
		+\sqrt{2n}H_{n-1}^{\a,\b}(x,t)[\a'(x-\b)-\a\b']\\
		&-[\a'(x-\b)-\a\b']\a(x-\b)H_n^{\a,\b}(x,t)\\
	\overset{\eqref{recurrence for RT hermite function}}
	=&\frac{\a'}{2\a}H_n^{\a,\b}(x,t)+\sqrt{n(n-1)}\frac{\a'}\a H_{n-2}^{\a,\b}(x,t)-\a\b'\sqrt{2n}H_{n-1}^{\a,\b}(x,t)+n\frac{\a'}\a H_n^{\a,\b}(x,t)\\
 &-\frac{\a'}{2\a}\sqrt{n(n-1)}H_{n-2}^{\a,\b}(x,t)+\a\b'\sqrt{\frac n2}H_{n-1}^{\a,\b}(x,t)-\frac{\a'}{2\a}(2n+1)H_n^{\a,\b}(x,t)\\
&+\a\b'\sqrt{\frac{n+1}2}H_{n+1}^{\a,\b}(x,t)-\frac{\a'}{2\a}\sqrt{(n+1)(n+2)}H_{n+2}^{\a,\b}(x,t)\\
	=&\frac{\a'}{2\a}\sqrt{n(n-1)}H_{n-2}^{\a,\b}(x,t)-\a\b'\sqrt{\frac n2}H_{n-1}^{\a,\b}(x,t)\\
	&+\a\b'\sqrt{\frac{n+1}2}H_{n+1}^{\a,\b}(x,t)-\frac{\a'}{2\a}\sqrt{(n+1)(n+2)}H_{n+2}^{\a,\b}(x,t).
\end{align*}
\end{proof}

\begin{remark}
The simple fact follows immediately from \eqref{eqn-derivative_t} and \eqref{orthogonal of H_n^alpha}. For $N\gg1$ sufficiently large, for $t\in[0,T]$, then
    \begin{align*}
        \left|\left|\partial_tH_N^{\alpha,\beta}(x,t)\right|\right|^2\lesssim
        \frac{\a'^2}{\a^2}N^2+\a'\b'N^{\frac32}+\a^2\b'^2N.
    \end{align*}
In particular, we have
\begin{equation*}
	\left|\left|\partial_tH_N^{\alpha,\beta}(x,t)\right|\right|^2\lesssim
        \frac{\a'^2}{\a^2}N^2.
\end{equation*}
\end{remark}

We follow the convection in the asymptotic analysis,
$a\sim b$ means that there exists some constants $C_1,C_2>0$ such
that $C_1a\leq b\leq C_2a$; $a\lesssim b$ means that there exists
some constant $C_3>0$ such that $a\leq C_3b$.

Any function $u(\cdot,t)\in L^2(\mathbb{R})$ can be written in the
form
\begin{align}\label{Hermite representation}
    u(\cdot,t)=\sum_{n=0}^\infty\hat{u}_n(t)H_n^{\alpha,\beta}(\cdot,t),\quad
    \hat{u}_n(t)=\int_{\mathbb{R}}u(x,t)H_n^{\alpha,\beta}(x,t)dx:=\left\langle u, H_n^{\a,\b}\right\rangle.
\end{align}
where $\left\{\hat{u}_n\right\}_{n=0}^\infty$ are the Fourier-Hermite
coefficients defined similarly as Fourier coefficients with
$H_n^{\alpha,\beta}(x,t)$ taking the place of harmonic oscillators.

For $N\geq0$, let
\begin{align*}
    \mathcal{R}_N(t)=\textup{span}\left\{H_0^{\alpha,\beta}(x,t),\cdots,H_N^{\alpha,\beta}(x,t)\right\}.
\end{align*}
At each time $t\in[0,T]$, it is a linear subspace of
$L^2(\mathbb{R})$. 

\begin{lemma}\label{lemma-estimate of L2 norm of phi in RN}
	For any function $\varphi(x,t)\in\CRN(t)$, we have
\begin{align*}
	||\p_x\varphi||\lesssim\a N^{\frac12}||\varphi||.
\end{align*}
\end{lemma}
\begin{proof}
	For any function $\varphi(x,t)\in\CRN(t)$, we can write it as
\begin{align*}
	\varphi(x,t)=\sum_{k=0}^N\hat{\varphi}_k(t)H_k^{\a,\b}(x,t).
\end{align*}
Thus,
\begin{align*}
	||\p_x\varphi||^2=&\langle\p_x\varphi,\p_x\varphi\rangle\\
	=&\a^2\sum_{k=0}^N\hat{\varphi}_k\left\langle-d(k+1)H_{k+1}^{\a,\b}+d(k)H_{k-1}^{\a,\b},\sum_{l=0}^N\hat{\varphi}_l
\left[-d(l+1)H_{l+1}^{\a,\b}+d(l)H_{l-1}^{\a,\b}\right]\right\rangle\\
	=&\sum_{k=0}^N\hat{\varphi}_k^2[d^2(k+1)+d^2(k)]
-2\a^2\sum_{k=2}^Nd(k-1)d(k)\hat{\varphi}_{k-2}\hat{\varphi}_k
	\lesssim\a^2N\sum_{k=0}^N\hat{\varphi}_k^2\leq\a^2N||\varphi||^2.
\end{align*}
\end{proof}

\subsection{Orthogonal projection and approximations}

Let us define the norm of the space $H_r(\R)$ as below:
\begin{definition} For any integer $r\geq0$, 
\begin{equation}\label{eqn-H_r}
	H_{r,\a}(\R):=\left\{u\in L^2(\R):\,||u||_r<\infty,
||u||_r^2:=\sum_{k=0}^\infty\l_{k+1}^r\hat{u}_k^2,\  \l_k=2\a^2k\right\}.
\end{equation}
\end{definition}

It is readily shown in \cite{XW} and \cite{XW2013} that when $\beta=0$ the norm defined in \eqref{eqn-H_r} is equivalent to the following Sobolev-like norm 
\begin{align*}
	||u||_{P,r,\a}^2=\sum_{k=0}^r\left|\left|\left(\a^4x^2+\a^2\right)^{\frac{r-k}2}\p_x^ku\right|\right|^2
\end{align*}
in the
sense that for any $u\in H_{r,\a}(\mathbb{R})$, with $r\geq0$
\begin{align*}
    ||u||_{r,\a}^2\sim\alpha||u||_{P,r,\a}^2.
\end{align*}
This equivalence can be extended trivially to where $\beta\neq0$, by mimicking the proof of Lemma 2.2-2.3 in \cite{XW} with $x$ replaced by $(x-\beta)$. 

From the definition of $||\cdot||_{P,r,\a}$, it is clear to see that the larger $r$ is, the
smaller space $H_{r,\a}$ is, and the smoother the functions are. The index $r$ can be viewed as the indicator of the regularity of the functions. 

Next we need some estimates on $||x^{\gamma}\partial_x^su(x)||^2$,
for any integer $\gamma,s\geq0$:
\begin{lemma}\label{lemma-seminorm estimate}
    For any function $u\in H_{s,\a}(\mathbb{R})$, with some integer $s\geq0$, we have
    \begin{align}\label{eqn-seminorm estimate}
        ||\partial_x^su||\lesssim||u||_{s,\a}.
    \end{align}
\end{lemma}
\noindent{\bf Proof.}\quad For any integer $s\geq0$,
\begin{align*}
    \left|\left|\partial_x^su\right|\right|^2
    =&\left|\left|\sum_{n=0}^\infty\hat{u}_n\partial_x^sH_n^{\alpha,\beta}\right|\right|^2
        =\left|\left|\sum_{n=0}^\infty\hat{u}_n\sum_{k=-s}^{s}
    a_{n,k}H_{n+k}^{\alpha,\beta}\right|\right|^2,
\end{align*}
where for each $n$ fixed, $a_{n,k}$ is a product
of $s$ factors of $\sqrt{\lambda_{n+j}}$ with $-s\leq j\leq s$.
Notice that $\lambda_{n+j}\sim\lambda_{n+1}$ for $n+j\geq0$ and
$H_{n+j}^{\alpha,\beta}\equiv0$ for $n+j<0$. Hence, we have
\begin{align*}
    \left|\left|\partial_x^su\right|\right|^2\overset{\eqref{orthogonal}}\lesssim\sum_{n=0}^\infty\lambda_{n+1}^{s}\hat{u}_n^2
        =||u||_{s,\a}^2,
\end{align*}
for any integer $s\geq0$.\hfill{$\Box$}

In the spectral method, a function in $L^2(\mathbb{R})$ is
approximated by the partial sum of the first $N$ frequency modes. We
define the $L^2$-orthogonal projection
$P_{N}^{\alpha,\beta}:\,L^2(\mathbb{R})\rightarrow\mathcal{R}_N(t)$, for some $t\in[0,T]$. Given $v\in L^2(\mathbb{R})$,
\begin{align*}
    \left\langle v-P_N^{\alpha,\beta}v,\phi\right\rangle=0,\quad\forall\,\phi\in\mathcal{R}_N(t).
\end{align*}
More precisely,
\begin{align*}
    P_Nv(x):=\sum_{n=0}^N\hat{v}_n(t)H_n^{\alpha,\beta}(x,t),
\end{align*}
for some $t\in[0,T]$, where $\hat{v}_n$ are the Fourier-Hermite coefficients defined in (\ref{Hermite representation}).

And the truncated error $||u-P_Nu||_{P,r,\a}$, for any integer
$r\geq0$, has been estimated readily in Theorem 2.3, \cite{GSX} for
$\alpha=1$, $\beta=0$, and in Theorem 2.1 , \cite{XW} for arbitrary
$\alpha>0$ and $\beta=0$. For arbitrary $\beta\neq0$, the estimate
is still valid. The proof is extremely similar to that of Theorem
2.1, \cite{XW}, except that we interpret the functions in the space
$H_{r,\a}$. For the sake of completeness, we include the proof here.
\begin{theorem}\label{theorem-truncation error}
    For any $u\in H_{r,\a}(\mathbb{R})$ and any integer $0\leq\mu\leq r$, we have
    \begin{align}\label{truncated error}
        \left|u-P_Nu\right|_\mu\lesssim\alpha^{\mu-r}N^{\frac{\mu-r}2}||u||_{r,\a},
    \end{align}
    where $|u|_\mu:=||\partial_x^\mu u||$ are the seminorms.
\end{theorem}
\begin{proof}
By induction, we first show that for
$\mu=0$. For any integer $r\geq0$,
\begin{align}\label{mu = 0}
    ||u-P_Nu||^2=\sum_{n=N+1}^\infty\hat{u}_n^2
        =\sum_{n=N+1}^\infty\lambda_{n+1}^{-r}\lambda_{n+1}^r\hat{u}_n^2
        \lesssim\alpha^{-2r}N^{-r}||u||_{r,\a}^2.
\end{align}
Suppose for $1\leq\mu\leq r$, (\ref{truncated error}) holds for
$\mu-1$. We need to show that (\ref{truncated error}) is also valid
for $\mu$. It is clear that
\begin{align}\label{split the seminorm into two parts}
    \left|u-P_Nu\right|_\mu\leq\left|\partial_xu-P_N\partial_xu\right|_{\mu-1}+\left|P_N\partial_xu-\partial_xP_Nu\right|_{\mu-1}.
\end{align}

On the one hand, due to the assumption for $\mu-1$, we apply
(\ref{truncated error}) to $\partial_xu$ and replace $\mu$ and $r$
with $\mu-1$ and $r-1$, respectively:
\begin{align}\label{part I of truncated error}
    |\partial_xu-P_N\partial_xu|_{\mu-1}\leq\alpha^{\mu-r}N^{\frac{\mu-r}2}||\partial_xu||_{r-1,\a}
        \lesssim\alpha^{\mu-r}N^{\frac{\mu-r}2}||u||_{r,\a},
\end{align}
where the last inequality holds with the observation that
\begin{align*}
    ||\partial_xu||_{r-1,\a}^2=\sum_{n=0}^\infty\lambda_{n+1}^{r-1}\widehat{(\partial_xu)}_n^2
\end{align*}
and
\begin{align*}
    \widehat{(\partial_xu)}_n=&\int_{\mathbb{R}}\partial_xuH_n^{\alpha,\beta}dx
        =-\int_{\mathbb{R}}u\partial_xH_n^{\alpha,\beta}dx
        \overset{\eqref{derivative_x}}=\frac{\sqrt{\lambda_{n+1}}}{2}\int_{\mathbb{R}}uH_{n+1}^{\alpha,\beta}dx
            -\frac{\sqrt{\lambda_n}}{2}\int_{\mathbb{R}}uH_{n-1}^{\alpha,\beta}dx,\\
        =&\frac{\sqrt{\lambda_{n+1}}}2\hat{u}_{n+1}-\frac{\sqrt{\lambda_n}}2\hat{u}_{n-1}.
\end{align*}
On the other hand, we have
\begin{align*}
    P_N\partial_xu-\partial_xP_Nu=&P_N\left(\sum_{n=0}^\infty\hat{u}_n\partial_xH_n^{\alpha,\beta}\right)-\sum_{n=0}^N\hat{u}_n\partial_xH_n^{\alpha,\beta}\\
        \overset{\eqref{derivative_x}}=&-\frac12\sum_{n=0}^{N-1}\sqrt{\lambda_{n+1}}\hat{u}_nH_{n+1}^{\alpha,\beta}
            +\frac12\sum_{n=0}^{N+1}\sqrt{\lambda_n}\hat{u}_nH_{n-1}^{\alpha,\beta}\\
&-\left[-\frac12\sum_{n=0}^N\sqrt{\lambda_{n+1}}\hat{u}_nH_{n+1}^{\alpha,\beta}
            +\frac12\sum_{n=0}^N\sqrt{\lambda_n}\hat{u}_nH_{n-1}^{\alpha,\beta}\right]\\
        =&\frac12\sqrt{\lambda_{N+1}}\left[\hat{u}_NH_{N+1}^{\alpha,\beta}+\hat{u}_{N+1}H_N^{\alpha,\beta}\right].
\end{align*}
This yields that
\begin{align}\label{part II of truncated error}
    \left|P_N\partial_xu-\partial_xP_Nu\right|_{\mu-1}^2 \lesssim
    \lambda_{N+1}\left(\hat{u}_N^2\left|H_{N+1}^{\alpha,\beta}\right|_{\mu-1}^2+\hat{u}_{N+1}^2\left|H_N^{\alpha,\beta}\right|_{\mu-1}^2\right),
\end{align}
due to the property of seminorms. Moreover, we estimate
$\hat{u}_k^2$ and $\left|H_k^{\alpha,\beta}\right|_{\mu-1}^2$, for
$k=N,N+1$:
\begin{align}\label{u_N}
    \hat{u}_N^2\leq\sum_{n=N}^\infty\hat{u}_n^2\leq||u-P_{N-1}u||^2\overset{\eqref{mu = 0}}\lesssim\alpha^{-2r}N^{-r}||u||_{r,\a}^2.
\end{align}
Similarly,
$\hat{u}_{N+1}^2\lesssim\alpha^{-2r}N^{-r}||u||_r^2$. And
\begin{align}\label{seminorm of H_N}
    \left|H_N^{\alpha,\beta}\right|_{\mu-1}^2
        =&\left|\left|\partial_x^{\mu-1}H_N^{\alpha,\beta}(x)\right|\right|^2
           \overset{\eqref{eqn-seminorm estimate}}\lesssim\left|\left|H_N^{\alpha,\beta}(x)\right|\right|_{\mu-1,\a}^2
	=\lambda_{N+1}^{\mu-1},
\end{align}
since $\widehat{(H_N^{\alpha,\beta})}_k=\delta_{kN}$, for
$k\in\mathbb{Z}^+$. Similarly,
$\left|H_{N+1}^{\alpha,\beta}\right|_{\mu-1}^2\lesssim\alpha^{-1}\lambda_{N+2}^{\mu-1}$.
Substitute (\ref{u_N}) and (\ref{seminorm of H_N}) into (\ref{part
II of truncated error}), we get
\begin{align}\label{part II of truncated error2}
    |P_N\partial_xu-\partial_xP_Nu|_{\mu-1}^2
         \lesssim\alpha^{2\mu-2r}N^{\mu-r}||u||_{r,\a}^2,
\end{align}
by the fact that $\lambda_N=2N\alpha^2$. Combine (\ref{split the
seminorm into two parts}), (\ref{part I of truncated error}) and
(\ref{part II of truncated error2}), we arrive the conclusion.
\end{proof}


\section{Time-dependent Galerkin Hermite spectral method (TGHSM) to nonlinear convection-diffusion equations}

\setcounter{equation}{0}

Let us consider the following nonlinear convection-diffusion equation
\begin{align}\label{1D KdVB with source}
    \left\{ \begin{aligned}
       u_t + a_1g(u)u_x-a_2 u_{xx} +a_3u_{xxx} &= f(x,t),\quad\textup{for}\
    (x,t)\in\mathbb{R}\times[0,T]\\
    u(x,0) &= u_0(x),
    \end{aligned} \right.
\end{align}
where $a_2\geq0$, $a_1,\ a_3$ are arbitrary real parameters, and $g(\cdot)\in C(\R)$ has the primitive function $G(\cdot)$.

The weak formulation of \eqref{1D KdVB with source} is 
\begin{align}\label{weak formulation}
    \left\{ \begin{aligned}
       \langle\partial_tu,\varphi\rangle - a_1\left\langle G\left(u\right),\p_x\varphi\right\rangle
       &+a_2\langle\partial_xu,\partial_x\varphi\rangle-a_3\langle\partial_{xx}u,\partial_x\varphi\rangle= \langle f,\varphi\rangle,\quad\textup{for}\
    (x,t)\in\mathbb{R}\times[0,T]\\
       u(x,0) &= u_0(x),
    \end{aligned} \right.
\end{align}

The THGSM of solving \eqref{1D KdVB with source} is to find $u_N\in\mathcal{R}_N(t)$ such
that
\begin{align}\label{HSM formulation}
    \left\{ \begin{aligned}
       \langle\partial_tu_N,\varphi\rangle - a_1\left\langle G\left(u_N\right),\p_x\varphi\right\rangle
       &+a_2\langle\partial_xu_N,\partial_x\varphi\rangle-a_3\langle\partial_{xx}u_N,\partial_x\varphi\rangle= \left\langle f,\varphi\right\rangle,\quad\textup{for}\
    (x,t)\in\mathbb{R}\times[0,T]\\
       u_N(x,0) &= P_N^{\a(0),\b(0)}u_0(x):=u_{N,0},
    \end{aligned} \right.
\end{align}
for any $\varphi\in\mathcal{R}_N(t)$.

We shall investigate the stability and convergence analysis under the following assumption\\[2pt]
\begin{enumerate}
	\item[] {\it Assumption 1:} $G(\cdot)$ has the primitive function $\tilde{G}(\cdot)$, and $\tilde{G}(0)=0$.\\[2pt]
	\item[] {\it Assumption 2:} $|g(x)|\lesssim 1+|x|^s$,and $|G(x)|\lesssim 1+|x|^{s+1}$, $\forall\,x\in\R$, for some $s\geq1$.\\[2pt]
\end{enumerate}

These two assumptions are not limited in the sense that they are satisfied by many important physical model equations, for example, the heat equation, the Burgers' equation and the Korteweg-de Vries Burger's (KdVB) equation. The first two equations are well-known and have been studied widely for a long time; while the last one is the model equation derived by Su and Gardner \cite{SG}, and first studied by \cite{S}, which describes a wide class of nonlinear systems in the weak nonlinearity and long wavelength approximations, since it contains both damping and dispersion. This model equation also has been used in the study of wave propagation through a liquid-filled elastic tube \cite{J1970} and for a description of shallow water waves on a viscous fluid \cite{J1972}. The existence and uniqueness of the global smooth solution of KdVB equation have been established in \cite{GBoling}. 
 
\subsection{A-priori estimates}

\begin{lemma}
    If $u_0\in L^2(\mathbb{R})$ and $f\in L^2\left([0,T];L^2(\mathbb{R})\right)$, and {\it Assumption 1} holds, then
\begin{align}\label{apriori lemma 1}
    \left|\left|u_N\right|\right|^2(t)+2a_2 e^T\int_0^t\left|\left|\partial_x u_N\right|\right|^2(s)ds\leq e^T\left(\left|\left|u_0\right|\right|^2+\int_0^t\left|\left|f\right|\right|^2(s)ds\right).
\end{align}
\end{lemma}
\begin{proof}
Take the test function $\varphi = 2u_N$
in (\ref{HSM formulation}), we have
\begin{align*}
    2\left\langle\partial_tu_N,u_N\right\rangle
-2a_1\left\langle G(u_N),\p_xu_N\right\rangle
+2a_2\left\langle\partial_xu_N,\partial_xu_N\right\rangle-2a_3\left\langle\partial_{xx}u_N,\partial_xu_N\right\rangle=2\left\langle f,u_N\right\rangle.
\end{align*}
Under {\it Assumption 1}, we have
\begin{align*}
	\left\langle G\left(u_N\right),\p_xu_N\right\rangle
=\left.\tilde{G}(u_N)\right|_{-\infty}^{\infty}=0,
\end{align*}
due to the fact that $\displaystyle\lim_{x\rightarrow\pm\infty}u_N(x,t)=0$, for all $t\in[0,T]$. And we also obtain that
\begin{align}\label{eqn-second derivative term}
	\left\langle\partial_{xx}u_N,\partial_xu_N\right\rangle
=\frac12\int_{\mathbb{R}}\partial_x[(\partial_xu_N)^2]dx=0,
\end{align}
since $u_N\in\CRN$, i.e., $\displaystyle\lim_{x\rightarrow\pm\infty}\p_x^ru_N(x,t)=0$, for any $r\in\N$ and $t\in[0,T]$.
That is,
\begin{align*}
    \frac d{dt}\left|\left|u_N\right|\right|^2+2a_2\left|\left|\partial_xu_N\right|\right|^2&=2\int_{\mathbb{R}}fu_Ndx\leq\left|\left|f\right|\right|^2+\left|\left|u_N\right|\right|^2,
\end{align*}
The Gronwall's inequality yields that
\begin{align*}
    \left|\left|u_N\right|\right|^2(t)+2a_2\int_0^te^{t-s}\left|\left|\partial_xu_N\right|\right|^2(s)ds
        \leq e^t\left|\left|u_N\right|\right|^2(0)+\int_0^te^{t-s}\left|\left|f\right|\right|^2(s)ds,
\end{align*}
which implies (\ref{apriori lemma 1}). 
\end{proof}

\subsection{Stability}

We shall consider the stability of \eqref{HSM formulation} in the sense of Guo \cite{G}, since it is impossible to prove the stability of \eqref{HSM formulation} in the sense of Courant et al. \cite{CFL}. Let us assume that $f$ and $u_{N,0}$ have the errors $\tf$ and $\tui$, respectively. They will introduce error in $u_N$, denoted as $\tuN$. The error satisfies the following equation
\begin{align}\label{eqn-error's equation}
	\left\{\begin{aligned}
		\left\langle\p_t\tuN,\varphi\right\rangle-a_1\left\langle G(u_N)-G(u_N+\tuN),\p_x\varphi\right\rangle+a_2\left\langle\p_x\tuN,\p_x\varphi\right\rangle-a_3\left\langle\p_{xx}\tuN,\p_x\varphi\right\rangle=&\left\langle\tf,\varphi\right\rangle\\
	\tuN(x,0)=&\tui,
	\end{aligned}\right.
\end{align}
where $(x,t)\in \R\times[0,T]$, for all $\varphi\in\CRN$. By taking $\varphi=2\tuN$ in \eqref{eqn-error's equation}, it yields that
\begin{align}\label{eqn-stability 1}
	\frac d{dt}||\tuN||^2-2a_1\left\langle g(u_N+\theta\tuN)\tuN,\p_x\tuN\right\rangle+2a_2||\p_x\tuN||^2=2\left\langle\tf,\tuN\right\rangle\leq\left|\left|\tf\right|\right|^2+\left|\left|\tuN\right|\right|^2,
\end{align}
due to the similar argument as in \eqref{eqn-second derivative term}. The second term on the left-hand side of \eqref{eqn-stability 1} can be easily estimated as
\begin{align}\label{eqn-estimate of g_triple}\notag
	|\langle g(u_N+\theta\tuN)\tuN,\p_x\tuN\rangle|
	\lesssim& \frac 1\epsilon||g(u_N+\theta\tuN)\tuN||^2+\epsilon||\p_x\tuN||^2\\\notag
	\lesssim& \frac 1\epsilon\left(1+||u_N||_\infty^s+||\tuN||_\infty^s\right)||\tuN||^2+\epsilon||\p_x\tuN||^2\\
	\lesssim& C(\epsilon,s,||u_N||_{L^\infty(0,T;L^\infty(\R))},||\tuN||_{L^\infty(0,T;L^\infty(\R))})||\tuN||^2+\epsilon||\p_x\tuN||^2,
\end{align}
where 
\[
	C(\epsilon,s,||u_N||_{L^\infty(0,T;L^\infty(\R)),||\tuN||_{L^\infty(0,T;L^\infty(\R))}})=\frac1\epsilon\left(||u_N||_{L^\infty(0,T;L^\infty(\R))}^s+||\tuN||_{L^\infty(0,T;L^\infty(\R))}^s\right).
\]
This constant is finite, due to the fact that $u_N$, $\tuN\in\CRN$, i.e., $||u_N||_\infty(t)$, $||\tuN||_\infty(t)<\infty$, for any $t\in[0,T]$. Therefore, we arrive the following stability result:
\begin{theorem}\label{thm-stability}
	Let $u_N$  be the solution of \eqref{HSM formulation} and $\tuN$ be error induced by the error of the source term $\tf$ and that of the initial condition $\tui$. Then for any $t\leq T$, we have
\begin{align*}
	||\tuN||^2+&C(a_1,a_2,\epsilon)\int_0^t||\p_x\tuN(s)||^2ds\\
	\lesssim& \rho\left(\tui,\tf,t\right)+C(a_1,s,||u_N||_{L^\infty(0,T;L^\infty(\R))},||\tuN||_{L^\infty(0,T;L^\infty(\R))},\epsilon)\int_0^t||\tuN(s)||^2ds,
\end{align*}
where 
\begin{align*}	
	C(a_1,a_2,\epsilon)=&2(a_2-a_1\epsilon)>0,\\[4pt]
	C(a_1,s,||u_N||_{L^\infty(0,T;L^\infty(\R))},&||\tuN||_{L^\infty(0,T;L^\infty(\R))},\epsilon)\\
	=&\frac {2a_1}\epsilon\left(||u_N||_{L^\infty(0,T;L^\infty(\R))}^s+||\tuN||_{L^\infty(0,T;L^\infty(\R))}^s\right)+1,
\end{align*}
and 
\[
	\rho\left(\tui,\tf,t\right)=||\tui||^2+\int_0^t\left|\left|\tf(s)\right|\right|^2ds.
\]
In particular, for all $t\leq T$,
\[
	||\tuN||^2+\int_0^t||\p_x\tuN(s)||^2ds\leq\rho\left(\tui,\tf,t\right)e^{C(a_1,u_N,T,\epsilon)t}.
\]
\end{theorem} 
This theorem implies that the error of the numerical solution is controlled by the errors of the initial data $\tui$ and the source term $\tf$. It means that \eqref{HSM formulation} is of generalized stability in the sense of Guo \cite{G}.

\subsection{Convergence analysis}

\begin{theorem}\label{thm-convergence}
    If $u_0\in H_r(\R)$, and {\it Assumption 1-2} are satisfied, then for any $u\in L^\infty(0,T;H_{r,\a(t)}(\R))\cap L^2(0,T;H_{r,\a(t)}(\R))$ with $r>\max\{2,\frac s4\}$ and $N\gg1$, we have
	\begin{enumerate}
		\item If $s< 8$, then
    \begin{align}\label{eqn-u-u_N_s<8}
        ||u-u_N||^2(t)\lesssim& \a^{-2r}N^{-r}||u||_{L^\infty(0,T;H_{r,\a(t)}(\R))}^2+CN^{2-r}\int_0^te^{(t-\tau)}||u(\tau)||_{r,\a(\tau)}^2d\tau;
    \end{align}
		\item If $s\geq8$, then
	\begin{align}\label{eqn-u-u_N_s>8}
		||u-u_N||^2(t)\lesssim& \a^{-2r}N^{-r}||u||_{L^\infty(0,T;H_{r,\a(t)}(\R))}^2+N^{\frac s4-r}\int_0^te^{(t-\tau)}\a^{\frac s2-2r}(\tau)||u(\tau)||_{r,\a(\tau)}^2d\tau
	\end{align}
	\end{enumerate}
    where $s$ is the growth rate of $g(\cdot)$ in {\it Assumption 2} and $C$ in \eqref{eqn-u-u_N_s<8} may depend on $\a$, $\a'$, $s$, $r$, $||u_N||_{L^\infty(0,T;L^2(\R))}$,  $||u||_{L^\infty(0,T;H_{r,\a(t)}(\R))}$ and $||u||_{L^2(0,T;H_{r,\a(t)}(\R))}$, etc.
\end{theorem}
\begin{proof}
Let $U_N=P_N^{\a,\b}u$ for simpler notation. By
the weak formulation (\ref{weak formulation}) of the nonlinear convection-diffusion equation \eqref{1D KdVB with source}, $U_N$ satisfies
\begin{align}\label{U_N's equation}
    \left\{ \begin{aligned}
        \langle\partial_tU_N,\varphi\rangle-&a_1\left\langle G\left(U_N\right),\p_x\varphi\right\rangle+a_2\langle\partial_xU_N,\partial_x\varphi\rangle-a_3\langle\partial_{xx}U_N,\partial_x\varphi\rangle\\
            =&-\left\langle\partial_t\left(u-U_N\right),\varphi\right\rangle+a_1\left\langle G(u)-G\left(U_N\right),\p_x\varphi\right\rangle\\
	&-a_2\left\langle\partial_x\left(u-U_N\right),\partial_x\varphi\right\rangle+a_3\left\langle\partial_{xx}\left(u-U_N\right),\partial_x\varphi\right\rangle+\langle f,\varphi\rangle\\
       U_N(x,t)=&P_N^{\a(0),\b(0)}u_0(x).
    \end{aligned} \right.
\end{align}
Let $\eta_N=u_N-U_N$. It satisfies the following equation, which is the difference of (\ref{HSM formulation}) and
(\ref{U_N's equation}):
\begin{align}\label{eta's equation}
    \left\{ \begin{aligned}
        \langle\partial_t\eta_N,\varphi\rangle&+a_2\langle\partial_x\eta_N,\partial_x\varphi\rangle-a_3\langle\partial_{xx}\eta_N,\partial_x\varphi\rangle\\
            =&\left\langle\partial_t\left(u-U_N\right),\varphi\right\rangle+a_1\left\langle G\left(u_N\right),\p_x\varphi\right\rangle-a_1\langle G(u)-G\left(U_N\right),\partial_x\varphi\rangle\\
             &+a_2\left\langle\partial_x\left(u-U_N\right),\partial_x\varphi\right\rangle
                -a_3\left\langle\partial_{xx}\left(u-U_N\right),\partial_x\varphi\right\rangle:=\sum_{i=1}^5G_i(\varphi)\\
       \eta_N(x,t)=&0,
    \end{aligned} \right.
\end{align}
for any $\varphi(\cdot,t)\in\mathcal{R}_N(t)$. Let us take the test function $\varphi=2\eta_N$, then
\begin{align}\label{LHS of eta's equation}
    \textup{L.H.S.\ of\ }(\ref{eta's equation})
        =&2\langle\partial_t\eta_N,\eta_N\rangle+2a_2\langle\partial_x\eta_N,\partial_x\eta_N\rangle-2a_3\langle\partial_{xx}\eta_N,\partial_x\eta_N\rangle
        \overset{\eqref{eqn-second derivative term}}=\frac d{dt}\left|\left|\eta_N\right|\right|^2+2a_2\left|\left|\partial_x\eta_N\right|\right|^2
\end{align}
Next, we estimate the right-hand side of (\ref{eta's equation}) term by term:
\begin{align*}
    \frac12G_1(2\eta_N)=&\left\langle\partial_t\left(u-U_N\right),\eta_N\right\rangle
       \overset{\eqref{eqn-derivative_t}}=\left\langle\partial_t\left[\hat{u}_{N+1}H_{N+1}^{\alpha,\beta}+\hat{u}_{N+2}H_{N+2}^{\a,\b}\right],\eta_N\right\rangle\\
	=&\left\langle\hat{u}_{N+1}\partial_tH_{N+1}^{\alpha,\beta},\eta_N\right\rangle+\left\langle\p_t\hat{u}_{N+1}H_{N+1}^{\a,\b},\eta_N\right\rangle\\
	&+\left\langle\hat{u}_{N+2}\partial_tH_{N+2}^{\alpha,\beta},\eta_N\right\rangle+\left\langle\p_t\hat{u}_{N+2}H_{N+2}^{\a,\b},\eta_N\right\rangle\\
     =&\left\langle\hat{u}_{N+1}\partial_tH_{N+1}^{\alpha,\beta},\eta_N\right\rangle+\left\langle\hat{u}_{N+2}\partial_tH_{N+2}^{\alpha,\beta},\eta_N\right\rangle,
\end{align*}
where the second and the last equalities are due to the fact that $\eta_N\in\mathcal{R}_N(t)$. Hence,  
\begin{align}
    \left|G_1\left(2\eta_N\right)\right|
        \leq&2\left|\hat{u}_{N+1}\right|\cdot\left|\left|\partial_tH_{N+1}^{\alpha,\beta}(x,t)\right|\right|\cdot\left|\left|\eta_N\right|\right|
	+2\left|\hat{u}_{N+2}\right|\cdot\left|\left|\partial_tH_{N+2}^{\alpha,\beta}(x,t)\right|\right|\cdot\left|\left|\eta_N\right|\right|\\\notag
      \overset{\eqref{eqn-derivative_t}}\lesssim&\frac{|\alpha'|}\alpha N\left(\left|\hat{u}_{N+1}\right|+\left|\hat{u}_{N+2}\right|\right)||\eta_N||.
\end{align}
Note the fact that
\begin{align*}
	\left|\hat{u}_{N+1}\right|\leq\left(\sum_{k=N+1}^\infty\hat{u}_k^2\right)^{\frac12}
	=\left(\sum_{k=N+1}^\infty\l_{k+1}^{-r}\l_{k+1}^r\hat{u}_k^2\right)^{\frac12}
	\leq\l_{N+2}^{-\frac r2}\left(\sum_{k=N+1}^\infty\l_{k+1}^r\hat{u}_k^2\right)^{\frac12}
	\lesssim\a^{-r}(t)N^{-\frac r2}||u||_{r,\a(t)}.
\end{align*}
Similarly, we have $\left|\hat{u}_{N+2}\right|\lesssim\a^{-r}(t)N^{-\frac r2}||u||_{r,\a(t)}$. We obtain the estimate of $\left|G_1(2\eta_N)\right|$:
\begin{align}\label{eqn-G_1}
	|G_1(2\eta_N)|\lesssim|\a'(t)|\a^{-(r+1)}(t)N^{1-\frac r2}||u||_{r,\a(t)}\cdot||\eta_N||
	\lesssim ||\eta_N||^2+|\a'(t)|^2\a^{-2-2r}(t)N^{2-r}||u||_{r,\a(t)}^2.
\end{align}
The second term $G_2(2\eta_N)$ on the right hand side of \eqref{eta's equation} can be estimated similarly as in \eqref{eqn-estimate of g_triple}:
\begin{align}\label{eqn-G_2}
	|G_2(2\eta_N)|=2a_1\left|\left\langle G(u_N),\p_x\eta_N\right\rangle\right|
	\lesssim\frac {a_2}2||\p_x\eta_N||^2+C||u_N||^s_{L^\infty(0,T;L^2(\R))}+C,
\end{align}
due to {\it Assumption 2}. The third term $G_3(2\eta_N)$ on the right hand side of \eqref{eta's equation} yields:
\begin{align*}
	|G_3(2\eta_N)|
	\lesssim|&\langle G(u)-G(U_N),\p_x\eta_N\rangle|
	=|\langle g(U_N+\theta(u-U_N))(u-U_N),\p_x\eta_N\rangle|\\
	\lesssim& C(1+||u||_\infty^s+||U_N||_\infty^s)||u-U_N||^2+\frac {a_2}2||\p_x\eta_N||^2
\end{align*}
By the interpolation inequality (cf. \cite{F}), we have
\begin{align*}
	||U_N||_\infty\lesssim||U_N||^{\frac12}||\p_xU_N||^{\frac12}
	\lesssim\a^{\frac12}N^{\frac14}||U_N||\leq\a^{\frac12}N^{\frac14}||u||,
\end{align*}
by Lemma \ref{lemma-estimate of L2 norm of phi in RN}, where $||\cdot||_\infty$ is the $L^\infty$ norm with respect to $x$. With this fact, we continue to estimate $G_3(2\eta_N)$:
\begin{align}\label{eqn-G_3}\notag
	|G_3(2\eta_N)|\lesssim&\frac{a_2}2||\p_x\eta_N||^2+C\a^{-2r}(t)N^{-r}\left(||u||_\infty^s+\a^{\frac s2}(t) N^{\frac s4}||u||^s\right)||u||_{r,\a(t)}^2\\
	\leq&\frac{a_2}2||\p_x\eta_N||^2+C\a^{\frac s2-2r}(t)N^{\frac s4-r}||u||_{r,\a(t)}^2.
\end{align}
Other than the estimate of $||U_N||_\infty$, by Theorem \ref{theorem-truncation error}, we have
\begin{align}\label{eqn-infty of U_N}
	||U_N||_\infty\lesssim||U_N||^{\frac12}||\p_xU_N||^{\frac12}
\lesssim\a^{\frac{1-r}2}(t)\left(1+N^{\frac{1-r}2}\right)^{\frac12}||u||^{\frac12}||u||_{r,\a(t)}^{\frac12}\lesssim\a^{\frac{1-r}2}(t)||u||_{r,\a(t)}^{\frac12}.
\end{align}
The rest two terms $G_i(2\eta_N)$, $i=4,5$, on the right-hand side of \eqref{eta's equation} can be easily bounded as follows:
\begin{align}\label{eqn-G_4}
	|G_4(2\eta_N)|\lesssim& \frac{a_2}2||\p_x\eta_N||^2+C||\p_x(u-U_N)||^2
	\leq\frac{a_2}2||\p_x\eta_N||^2+C\a^{2-2r}(t)N^{1-r}||u||_{r,\a(t)}^2;
\end{align}
and 
\begin{align}\label{eqn-G_5}
	|G_5(2\eta_N)|\lesssim\frac{a_2}2||\p_x\eta_N||^2+C||\p_{xx}(u-U_N)||^2
	\lesssim\frac{a_2}2||\p_x\eta_N||^2+C\a^{4-2r}(t)N^{2-r}||u||_{r,\a(t)}^2.
\end{align}
Substituting \eqref{eqn-G_1}-\eqref{eqn-G_3}, \eqref{eqn-G_4} and \eqref{eqn-G_5} into the right-hand side of \eqref{eta's equation} and combining \eqref{LHS of eta's equation}, we obtain that
\begin{align*}
	\frac d{dt}\left|\left|\eta_N\right|\right|^2
\lesssim& ||\eta_N||^2+\left(CN^{2-r}+\a^{\frac s2-2r}(t)N^{\frac s4-r}\right)||u||_{r,\a(t)}^2,
\end{align*}
where $C$ is the function of $t$ dependent of $\a$, $\a'$, $s$, $r$, $||u_N||_{L^\infty(0,T;L^2(\R))}$, etc. The results \eqref{eqn-u-u_N_s<8} and \eqref{eqn-u-u_N_s<8} are obtained by the Gronwall's inequality and Theorem \ref{theorem-truncation error}. 
\end{proof}

\section{Numerical results} 

\setcounter{equation}{0}

In this section, we shall use our THGSM \eqref{weak formulation} to numerically solve three benchmark equations all with $g(u)=\frac u2$ in \eqref{1D KdVB with source}, including the heat equation, the Burgers' equation and the KdVB equation with different coefficients $a_i$, $i=1,2,3$ in \eqref{1D KdVB with source}. To illustrate the high accuracy of our method, we shall compare our numerical results of the first two equations with the ones obtained by using the similarity transformation technique \cite{FK} and by using the stabilized Hermite spectral method developed in \cite{MST}. 

In our scheme \eqref{HSM formulation}, we choose the test function $\varphi=H_m^{\a,\b}(x,t)\in\CRN(t)$, for $m=0,1,\cdots,N$. It yields a system of ordinary differential equations (ODE) for $\hat{u}_m(t)$, $m=0,1,\cdots,N$. The second term on the left-hand side and the one on the right-hand side of \eqref{HSM formulation} need to be computed numerically; while all the other terms can be explicitly written as sparse matrices, which are followed from the properties of the GHF. To be more precise, the first term gives
\begin{align*}
	\frac d{dt}\vec{\hat{u}}+A_1\vec{\hat{u}}
\end{align*} 
where $\vec{\hat{u}}=[\hat{u}_0,\hat{u}_1,\cdots,\hat{u}_N]^T$, the superscript $T$ means the transpose, where the matrix $A_1$ is explicitly given by \eqref{eqn-derivative_t} as:
\begin{align}\label{eqn-A_1}
	A_1(i,j)=\left\{\begin{aligned}
	-\frac{\a'}\a d(i)d(i-1),\quad&\textup{if}\quad i=j+2\\
	\a\b' d(i),\quad&\textup{if}\quad i=j+1\\
	-\a\b' d(j),\quad&\textup{if}\quad i=j-1\\
	\frac{\a'}\a d(j)d(j-1),\quad&\textup{if}\quad i=j-2\\
	0,\quad&\textup{otherwise}
	\end{aligned},\right.
\end{align}
$d(n)=\sqrt{\frac n2}$ as before. The third term on the left-hand side of \eqref{HSM formulation} gives $a_2A_3\vec{\hat{u}}$, where 
\begin{align}\label{eqn-A_3}
	A_3(i,j)=\left\{\begin{aligned}
		\a^2(t)[d^2(i+1)+d^2(i)],\quad&\textup{if}\
    i=j\\
        -\a^2(t)d(l+1)d(l+2),\quad& l=\min\{i,j\},\ \textup{if}
    \ |i-j|=2\\
       0,\quad&\textup{otherwise}
	\end{aligned},\right.
\end{align}
 by \eqref{orthogonality of derivative}. And the last term on the left-hand side of \eqref{HSM formulation} yields $-a_3A_4\vec{\hat{u}}$, where
\begin{align}\label{eqn-A_4}
	A_4(i,j)=\left\{\begin{aligned}
		\a^2(t)d(j+1)d(j+2)d(j+3),\quad&\textup{if}\ i=j+3\\
		-\a^2(t)\left[d^3(j+1)+d^2(j)d(j+1)+d(j+1)d^2(j+2)\right],\quad&\textup{if}\ i=j+1\\
		\a^2(t)\left[d^3(i+1)+d^2(i)d(i+1)+d(i+1)d^2(i+2)\right],\quad&\textup{if}\ i=j-1\\
		-\a^2(t)d(i+1)d(i+2)d(i+3),\quad&\textup{if}\ i=j-3\\
		0,\quad&\textup{otherwise}
	\end{aligned}.\right.
\end{align}
At last, the THGSM \eqref{HSM formulation} yields the ODE system for the Fourier-Hermite coefficients $\vec{\hat{u}}$:
\begin{align}\label{eqn-ODE}
	\frac {d\vec{\hat{u}}}{dt}+A\vec{\hat{u}}+B\left(\vec{\hat{u}}\right)=\vec{\hat{f}},
\end{align}
where $B\left(\vec{\hat{u}}\right)$ is given by the second term $\frac{a_1}2\langle\p_x\left(u^2\right),\varphi\rangle$ of the left-hand side of \eqref{HSM formulation}, $\vec{\hat{f}}$ is the Fourier-Hermite coefficients of the source term $f(x,t)$ and $A=A_1+A_3+A_4$, $A_1$, $A_3$ and $A_4$ are given in \eqref{eqn-A_1}-\eqref{eqn-A_4}, respectively. $\vec{\hat{f}}$ and $B\left(\vec{\hat{u}}\right)$ need to be computed numerically. The equation \eqref{eqn-ODE} will be solved by using the Crank-Nicolson/Euler forward scheme in all the following examples. In particular, the Crank-Nicolson scheme is applied to all the linear terms, while the Euler forward scheme is applied to the nonlinear term $B\left(\vec{\hat{u}}\right)$.

In all the numerical experiments below, we shall use the $L^2$-norm and the relative $L^\infty$-norm to measure the accuracy. In our context, let us denote the $L^2$ error and the relative $L^\infty$ error as
\[
	E_N(t)=\left|\left|u_N(t)-u_{exact}(t)\right|\right|
\]
and 
\[
 E_{N,\infty}(t)=\frac{\max_{0\leq j\leq N}|u_N(y_j,t)-u_{exact}(y_j,t)|}{\max_{0\leq j\leq N}|u_{exact}(y_j,t)|},
\]
respectively, where $\{y_j\}_{0\leq j\leq N}$ are the Hermite-Gauss points. For comparison purpose, the examples in section 4.1 and 4.2 are taken from \cite{FK}, \cite{GX} and \cite{MST}. The definitions of $E_N(t)$ and $E_{N,\infty}(t)$ are slightly different in various schemes. For example, the errors of scheme in \cite{FK} are in the weighted $L^2$-norm after the transformation, see definition on pp. 615, \cite{FK}; those of scheme in \cite{MST} are in the weighted $L^2$-norm, see details on pp. 71, \cite{MST}. 

\subsection{Heat equation with source term}

We consider \eqref{1D KdVB with source} with $a_1=a_3=0$, $a_2=1$, $g(u)=\frac u2$ and the following source term:
\[
	f(x,t)=[x\cos{x}+(t+1)\sin{x}](t+1)^{-\frac32}e^{-\frac{x^2}{4(t+1)}},
\]
which yields the heat equation.This equation with the initial condition $u(x,0)=\sin{(x)}e^{-\frac{x^2}4}$ has been investigated in \cite{FK}. And its exact solution is 
\[
	u_{exact}(x,t)=\frac{\sin{x}}{\sqrt{t+1}}e^{-\frac{x^2}{4(t+1)}}.
\]
It is easy to see from the exact solution that it concentrates near the origin, so the translating factor $\b$ can be chosen to be zero. Moreover, the choice of the scaling factor $\a$ is extremely important of improving the accuracy of our scheme \eqref{HSM formulation}. We refer the readers to \cite{LY} for the guidelines to pick appropriate scaling factor. In this example, we shall investigate our scheme with the time-invariant scaling factor $\a=\frac{\sqrt{2}}2$ and the time-dependent one $\a(t)=\frac1{\sqrt{2(t+1)}}$. The HSM with time-invariant scaling factor is essentially developed in \cite{XW,LY} and can be viewed as a special case of our THGSM. The accuracy of the numerical results obtained by our THGSM and the time-invariant scheme are compared with some previous results, including the ones in \cite{FK}, \cite{MST}. 

\begin{table}[!hbp]
	\begin{tabular}{|c|c|c|c|c|}
\hline
	Time step&Funaro and Kavian's&Ma, Sun and Tang's& HSM&our scheme \eqref{HSM formulation} \\
$dt$& scheme \cite{FK}&scheme \cite{MST} &  scheme \cite{XW,LY}& with $\a(t)=\frac1{\sqrt{2(t+1)}}$\\
&&with $(\delta_0,\delta)=(1,1)$  &with $\a=\frac{\sqrt{2}}2$&\\
\hline 
	$250^{-1}$&$2.478$E-$03$&$2.958$E-$04$&$9.0032$E-$07$&$2.4045$E-$07$\\
\hline
	$1000^{-1}$&$6.203$E-$04$&$1.189$E-$06$&$7.6286$E-$07$&$4.8534$E-$08$\\
\hline
	$4000^{-1}$&$1.550$E-$04$&$1.177$E-$06$&$7.6213$E-$07$&$4.6247$E-$08$\\
\hline
	$16000^{-1}$&$3.886$E-$05$&$1.177$E-$06$&$7.6212$E-$07$&$4.6238$E-$08$\\
\hline
\end{tabular}
\bigskip
	\caption{\small{Errors of the heat equation at $T=1$ with $N=20$ using different methods.}}\label{table-heat eqn}
\end{table}

The results are list in Table \ref{table-heat eqn}. It is clear to see that our scheme is superior to the schemes in \cite{FK} and \cite{MST} in accuracy. And it is as we expected that the time-dependent scaling factor yields better resolution than the constant one, i.e. the scheme in \cite{XW,LY}. From the viewpoint of computational complexity, our scheme can be implemented straight-forward without any variable transformation. And the matrix $A$ in \eqref{eqn-ODE} is tri-diagonal, symmetric, and diagonally dominant, which can be accurately and effectively inversed (part of the computation in Crank-Nicolson scheme).

To show the rate of convergence of \eqref{HSM formulation}, we list in Table \ref{table-heat eqn convergence rate} the numerical errors at $T=1$ with various time steps $dt$ and the truncation modes $N$. It confirms the theoretical prediction that the scheme \eqref{HSM formulation} is of second-order accuracy in time and spectral accuracy in space. 

\begin{table}[!htp]
	\begin{tabular}{|c|c|c|c|c|}
\hline	
	Time step $dt$&$N$&$E_N(1)$&$E_{N,\infty}(1)$&Order\\
\hline
	$1$E-$1$&\multirow{4}{*}{$40$}&$1.7439$E-$04$&$1.2773$E-$04$&\\
	$1$E-$2$&&$1.7473$E-$06$&$1.2757$E-$06$&$dt^{2.0005}$\\
	$1$E-$3$&&$1.7473$E-$08$&$1.2757$E-$08$&$dt^{2.0002}$\\
	$1$E-$4$&&$1.7478$E-$10$&$1.2759$E-$10$&$dt^{2.0001}$\\
\hline
\end{tabular}
\\[5pt]
\begin{tabular}{|c|c|c|c|c|}
\hline	
	Time step $dt$&$N$&$E_N(1)$&$E_{N,\infty}(1)$&Order\\
\hline
	\multirow{3}{*}{$1$E-$4$}&$8$&$7.4$E-$03$&$9$E-$03$&\\
	&$16$&$4.2446$E-$06$&$4.7275$E-$06$&$N^{-10.77}$\\
	&$32$&$1.6540$E-$10$&$1.3012$E-$10$&$N^{-12.71}$\\
\hline
	\end{tabular}
\bigskip
	\caption{\small{Error of the heat equation by using the proposed scheme \eqref{HSM formulation} with different time steps $dt$ and truncation mode $N$.}}\label{table-heat eqn convergence rate}
\end{table}

\subsection{Viscous Burger's equation}

Besides the linear problems, our scheme is also effective in nonlinear problems. Let us take the viscous Burgers' equation as an example. The inviscid Burgers' equation suffers from the Gibbs' phenomenon, due to the lack of the diffusion term, i.e. $a_2=0$. The vanishing viscosity method combined with the spectral method, and the post-pocessing of the numerical solution are expected, see discussion in \cite{Ta,Ta1990,L} and reference therein. The viscous Burgers' equation \eqref{1D KdVB with source} with $a_1=a_2=1$, $a_3=0$ and $g(u)=\frac u2$, has been studied in \cite{GX} by the variable transformation technique
 \[
	y=\frac x{2\sqrt{t+1}},\quad s=\ln{(t+1)}
\]
 that a soliton-like solution 
\[
	u_{exact}(x,t) = e^{-\frac{x^2}{4(1+t)}}\sech^2\left(\frac {ax}{2(1+t)}-b\ln{(1+t)}-c\right)
\]
is obtained with $a=0.3$, $b=0.5$ and $c=-3$, where the source term is 
\begin{align*}
	f(x,t) =& -e^{-\frac{x^2}{2(1+t)}}\frac{\sech^4(\xi)}{1+t}{\left[\frac x2+a\,\tanh{\xi}\right]}\\
	&+e^{-\frac{x^2}{4(1+t)}}\frac{\sech^2(\xi)}{1+t}\left[\frac{1+t+a^2}{2(1+t)}+2b\,\tanh(\xi)-\tanh^2(\xi)\frac{3a^2}{2(1+t)}\right],
\end{align*}
$\xi=\frac {ax}{2(1+t)}-b\ln{(1+t)}-c$.

In this example, the second term on the left-hand side of \eqref{HSM formulation} is computed numerically. By taking $\varphi=H_m^{\a,\b}$, $m=0,\cdots,N$, the key component of this term 
\[
	\left\langle \p_x\left(u_N^2\right),\varphi\right\rangle=-\left\langle u_N^2,\p_xH_m^{\a,\b}\right\rangle
	\overset{\eqref{derivative_x}}=-\a d(m+1)\left\langle u_N^2,H_{m+1}^{\a,\b}\right\rangle+\a d(m)\left\langle u_N^2,H_{m-1}^{\a,\b}\right\rangle.
\]
follows immediately with the numerical integral 
\[
	\int_{\R}H_l^{\a,\b}H_n^{\a,\b}H_m^{\a,\b},
\]
$l,n,m=0,\cdots,N$. We shall use the Gauss-Hermite quadrature rule to perform this integration.

In Table \ref{table-viscous Burger's eqn compared with other schemes}, we compare our time-dependent scheme with Guo-Xu's scheme in \cite{GX} and the time-dependent scheme in \cite{MST}. In this numerical experiment, we choose the same time step $dt=(e-1)\times10^{-3}$ and the total experimental time $T=e-1$ as in \cite{GX} and \cite{MST} for comparison. It is clear to see that our scheme is at least as good as the one in \cite{MST}, and is more accurate than the one in \cite{GX}. 

Moreover, in Table \ref{table-viscous Burger's eqn different dt} we compare our scheme with the one in \cite{MST} with different time steps $dt$. We choose the total experimental time $T=1$, and the truncation modes $N=40$ and $20$, for the scheme in \cite{MST} and our THGSM, respectively. It is shown that they are almost as the same accuracy. In this table, we have also verified the first-order accuracy in time, due to the Euler forward scheme used in the nonlinear term $B\left(\vec{\hat{u}}\right)$. In Table \ref{table-viscous Burger's eqn spectral accuracy}, we illustrate the spectral accuracy of our scheme in space as we showed in Theorem \ref{thm-convergence}.

\begin{table}[!htp]
	\begin{tabular}{|c|c|c|c|}
\hline
	$N$&Guo and Xu's result \cite{GX}&Ma, Sun and Tang's scheme \cite{MST}& our scheme \eqref{HSM formulation}\\
\hline
	$8$&$1.381$E-$06$&$1.563$E-$05$&$1.7669$E-$06$\\
\hline
	$16$&$1.381$E-$06$&$6.337$E-$07$&$1.1516$E-$07$\\
\hline
	$32$&$1.381$E-$06$&$1.031$E-$07$&$1.3400$E-$07$\\
\hline	
	\end{tabular}
\bigskip
	\caption{\small{Errors of viscous Burger's equation at $T=e-1$ with time step $dt = T*10^{-3}$. The error of scheme in \cite{GX} and that in \cite{MST} are $E_N(T)$ in their contexts, see details on pp. 869, \cite{GX} and  on pp. 71, \cite{MST}, respectively. The one in our scheme \eqref{HSM formulation} is $E_N(T)$. }}\label{table-viscous Burger's eqn compared with other schemes}
\end{table}
\begin{table}[!htp]
	\begin{tabular}{|c|c|c|c|c|c|c|}
\hline
		\multirow{2}{*}{Time step $dt$}&\multicolumn{3}{c|}{Ma, Sun and Tang's scheme \cite{MST}}&\multicolumn{3}{c|}{time-dependent HSM \eqref{HSM formulation}}\\
\cline{2-7}
&$E_{40}(1)$&$E_{40,\infty}(1)$&Order&$E_{20}(1)$&$E_{20,\infty}(1)$&Order\\
\hline
$1$E-$1$&$5.101$E-$04$&$4.677$E-$03$&&$4.8044$E-$06$&$1.4264$E-$04$&\\
$1$E-$2$&$4.508$E-$06$&$4.548$E-$05$&$dt^{2.05}$&$4.1512$E-$07$&$1.0680$E-$05$&$dt^{1.06}$\\
$1$E-$3$&$4.454$E-$08$&$4.530$E-$07$&$dt^{2.01}$&$4.1065$E-$08$&$1.0324$E-$06$&$dt^{1.03}$\\
$1$E-$4$&$4.467$E-$10$&$4.372$E-$09$&$dt^{2.00}$&$4.1771$E-$09$&$9.7699$E-$08$&$dt^{1.02}$\\
\hline
	\end{tabular}
\bigskip
\caption{\small{Errors of viscous Burger's equation by using scheme in \cite{MST} and our scheme \eqref{HSM formulation} with different time steps. The errors $E_{N}(t)$ and $E_{N,\infty}(t)$ in the scheme \cite{MST} and our scheme are defined slightly different.}}\label{table-viscous Burger's eqn different dt}
\end{table}
\begin{table}[!htp]
\begin{tabular}{|c|c|c|c|c|}
\hline	
	Time step $dt$&$N$&$E_N(1)$&$E_{N,\infty}(1)$&Order\\
\hline
	\multirow{3}{*}{$1$E-$5$}&$5$&$2.4254$E-$06$&$9.4692$E-$05$&\\
	&$15$&$1.1291$E-$08$&$4.5015$E-$07$&$N^{-4.88}$\\
	&$25$&$4.3637$E-$10$&$1.0527$E-$08$&$N^{-5.28}$\\
\hline
\end{tabular}
\bigskip
\caption{\small{Errors of viscous Burger's equation by using our scheme \eqref{HSM formulation} with different truncation modes $N$.}}\label{table-viscous Burger's eqn spectral accuracy}
\end{table}

\subsection{KdVB equation}

In this subsection, we shall reveal the effect of the time-dependent translating factor, which has been concealed in the theoretical proof of the convergence analysis. Let us consider the 1D KdVB equation \eqref{1D KdVB with source} with $a_1=a_2=1$, $a_3=-\frac1{16}$, $g(u)=\frac u2$ and the source term
\[
	f(x,t)=-8\sech^2{\xi}\left(2-3\sech^2{\xi}+2\tanh{\xi}\,\sech^2{\xi}\right   ),
\]
where $\xi=2(x+t)$. It is easy to verify that 
\[
	u_{exact}(x,t)=\sech^2(\xi)
\]
is the exact solution to the above KdVB equation. The Crank-Nicolson/Euler forward scheme is used as in the case of the viscous Burgers' equation. 

In Table \ref{table-kdVB eqn_time steps}, we list the errors of numerical solutions to the KdVB equation with various time steps. The total experimental time is $T=1$, and the truncation mode is $N=40$. The scaling factor is chosen to be a constant $\a=2\sqrt{2}$ and the translating factor is $\b=0$. It is shown that we can achieve almost the first-order accuracy in time. However, the approximate solution is less accurate than the one we obtained in the viscous Burgers' equation. The possible reason for this may be the drifting to the left of the exact solution.

To tackle this problem, we move our GHF accordingly by setting a time-dependent translating factor $\b=-t$. We compare the results with those obtained without translating factor, i.e. $\b=0$. In Table \ref{table-KdVB eqn convergence rate}, we display the errors of the scheme \eqref{HSM formulation} with and without translating factor. The total experimental time is $T=1$, and the scaling factor is chosen to be a constant $\a=2\sqrt{2}$. Various truncation modes $N=10,20,\cdots,50$ are numerically experimented. It is shown in Table \ref{table-KdVB eqn convergence rate} that the time-dependent translating factor yields better resolution. The spectral accuracy is verified in the scheme \eqref{HSM formulation} for the KdVB equation in both with/without translating factor cases.
\begin{table}[!htp]
	\begin{tabular}{|c|c|c|c|}
\hline	
	Time step $dt$&$E_{40}(1)$&$E_{40,\infty}(1)$&Order\\
\hline
	$1$E-$2$&$3.11$E-$02$&$1.6$E-$02$&\\
	$1$E-$3$&$4.4076$E-$04$&$2.83$E-$04$&$dt^{1.85}$\\
	$1$E-$4$&$3.4747$E-$04$&$2.9070$E-$04$&$dt^{0.98}$\\
\hline
\end{tabular}
\bigskip
	\caption{\small{Errors of KdVB equation with different time steps.}}\label{table-kdVB eqn_time steps}
\end{table}
\begin{table}[!htp]
\begin{tabular}{|c|c|c|c|c|c|c|c|}
\hline	
	\multirow{2}{*}{Time step $dt$}&\multirow{2}{*}{$N$}&\multicolumn{3}{c|}{No translating factor $\beta=0$}&\multicolumn{3}{c|}{Time-dependent translating factor $\beta=-t$}\\
\cline{3-8}
&&$E_N(1)$&$E_{N,\infty}(1)$&Order&$E_N(1)$&$E_{N,\infty}(1)$&Order\\
\hline
	\multirow{5}{*}{$1$E-$4$}&$10$&$9.35$E-$02$&$7.54$E-$02$&&$5.7$E-$03$&$2.6$E-$03$&\\
	&$20$&$9.6$E-$03$&$8.2$E-$03$&$N^{-3.28}$&$2.8944$E-$04$&$1.2570$E-$04$&\\
	&$30$&$1.6$E-$03$&$1.4$E-$03$&$N^{-3.66}$&$6.8894$E-$05$&$2.5007$E-$05$&$N^{-4.05}$\\
	&$40$&$3.4747$E-$04$&$2.9070$E-$04$&$N^{-4.17}$&$1.3225$E-$05$&$4.5626$E-$06$&$N^{-4.27}$\\
	&$50$&$9.0630$E-$05$&$7.1222$E-$05$&$N^{-4.36}$&$3.5135$E-$06$&$1.1749$E-$06$&$N^{-4.49}$\\
\hline
	\end{tabular}
\bigskip
	\caption{\small{Errors of the KdVB equation by using the proposed scheme \eqref{HSM formulation} without translating factor and with time-dependent translating factor $\b=-t$. Different truncation modes $N$ are numerically experimented.}}\label{table-KdVB eqn convergence rate}
\end{table}

\section{Conclusions}

In this paper, we investigate a time-dependent Hermite-Galerkin spectral method (THGSM) to solve the nonlinear convection-diffusion equations in the whole line. Many important physical model equations are within the framework in this paper. Our THGSM is formulated under the basis of the generalized Hermite functions (GHF), which includes the time-dependent scaling factor and translating factor in the definition. The advantages of this formulation are at least two folds: on the one hand, the proof of the stability and convergence analysis are much easier than the previous Hermite spectral method, say in \cite{FK} and \cite{MST}, where they analyze in the weighted $L^2$ space. On the other hand, the implementation of our method is straight-forward. No variable transformation techinique is required. Furthermore, the derived ODE system in the linear equation is with extremely low computational cost, since the stiffness matrix is symmetric. The numerical experiments are carried out in some benchmark examples, including the heat equation, the viscous Burgers' equation and the KdVB equation. It is clear to see that our method surpasses nearly all the existing methods in accuracy.

\end{document}